\documentclass[12pt]{article}
\usepackage[utf8]{inputenc}
\usepackage{xcolor}
\usepackage{amssymb}
\usepackage{enumitem}
\usepackage{amsmath}
\usepackage{amsthm}
\usepackage{hyperref}
\usepackage{fullpage}
\usepackage{graphicx}

\DeclareMathOperator{\im}{Im}
\newcommand{\eps}{\varepsilon}
\renewcommand{\ge}{\geqslant}
\renewcommand{\le}{\leqslant}

\newcommand{\N}{\mathbb{N}}
\renewcommand{\c}{\mathcal}
\newcommand{\cA}{\ensuremath{\mathcal A}}
\newcommand{\cR}{\ensuremath{\mathcal R}}
\newcommand{\cC}{\ensuremath{\mathcal C}}
\newcommand{\cH}{\ensuremath{\mathcal H}}
\newcommand{\cQ}{\ensuremath{\mathcal Q}}

\theoremstyle{plain}
\newtheorem{theorem}{Theorem}[section]
\newtheorem{conjecture}[theorem]{Conjecture}
\newtheorem{proposition}[theorem]{Proposition}
\newtheorem{lemma}[theorem]{Lemma}
\theoremstyle{definition}
\newtheorem{definition}[theorem]{Definition}
\newtheorem{claim}{Claim}[theorem]
 
\title{Embedding loose spanning trees in $3$-uniform hypergraphs}
\author{Yanitsa Pehova\thanks{Department of Mathematics, London School of Economics, WC2A 2AE London, United Kingdom. Email: \href{mailto:y.pehova@lse.ac.uk}{\nolinkurl{y.pehova@lse.ac.uk}}. This author was supported by the Engineering and Physical Sciences Research Council, UK Research and Innovation [grant number EP/V038168/1].} \and Kalina Petrova \thanks{Department of Computer Science, ETH, 8092 Z\"urich, Switzerland.
			Email: \href{mailto:kalina.petrova@inf.ethz.ch}{\nolinkurl{kalina.petrova@inf.ethz.ch}}. This author was supported by grant no. CRSII5 173721 of the Swiss National Science Foundation.}}
\date{}

\begin{document}

\maketitle

\begin{abstract}
In 1995, Koml\'os, S\'ark\"ozy and Szemer\'edi showed that every large $n$-vertex graph with minimum degree at least $(1/2 + \gamma)n$ contains all spanning trees of bounded degree. We consider a generalization of this result to loose spanning hypertrees in $3$-graphs, that is, linear hypergraphs obtained by successively appending edges sharing a single vertex with a previous edge. We show that for all $\gamma$ and $\Delta$, and $n$ large, every $n$-vertex $3$-uniform hypergraph of minimum vertex degree $(5/9 + \gamma)\binom{n}{2}$ contains every loose spanning tree $T$ with maximum vertex degree $\Delta$. This bound is asymptotically tight, since some loose trees contain perfect matchings.
\end{abstract}

\section{Introduction}

Given a graph, finding conditions that guarantee the existence of various spanning subgraphs is a widely studied area of extremal combinatorics. A prime example of a condition of this type was famously given by Dirac in 1952, who showed that every $n$-vertex graph with minimum degree at least $n/2$ contains a Hamilton cycle. Another classical result in the same spirit, by Koml\'os, S\'ark\"ozy and Szemer\'edi \cite{kss1995}, states that for large $n$, any $n$-vertex graph with minimum degree $(1/2+\gamma)n$ contains every spanning tree of bounded degree. The constant $1/2$ is best possible as can be seen by considering a graph that consists of two cliques on $n/2$ vertices each. This result was later generalised in multiple directions: the same set of authors showed that the same holds for spanning trees of maximum degree $c n/\log n$ for a small constant $c$ \cite{kss2001}, and Csaba, Levitt, Nagy-Gy\"orgy, and Szemer\'edi \cite{csaba2010} showed that the minimum degree condition can be relaxed to $n/2+C\log n$ for a large constant $C$. This question has also been studied in the directed setting \cite{mycroft_naia_1,mycroft_naia_2,kathapurkar22}. 

In this paper we consider the corresponding spanning tree problem for hypergraphs. There is no single ``correct'' way to define a hypertree. An early definition due to Kalai \cite{kalai83} (also see \cite{linial19}) is in terms of simplicial homology. 
A $d$-hypertree (which we call \emph{a simplicial $d$-hypertree} to distinguish it from other hypertrees) is defined as a $d$-dimensional simplicial complex with a full $(d-1)$-skeleton whose $d$-th and $(d-1)$-st reduced rational homology groups vanish. 
A $2$-uniform tree is then a simplicial $1$-hypertree.
For $d\ge2$, such hypertrees are very dense and can be avoided in hypergraphs with arbitrarily large minimum degree (see Proposition~\ref{prop:3ap_example}). 
Motivated by this, we work with a combinatorial definition of a \emph{$k$-uniform $\ell$-tree} --- a $k$-uniform hypergraph admitting an edge ordering $e_1,...,e_m$ such that each $e_i$ consists of $\ell$ vertices in one previous edge in the ordering and $k-\ell$ new vertices\footnote{More formally, for each $i\ge 2$ there exists $j<i$ such that $e_i\cap \bigcup_{j'<i}e_{j'}\subseteq e_j$ and $|e_i\cap e_j|=\ell$.}. Such orderings we call \emph{valid}, and the edges which can be last in a valid ordering we call \emph{leaves}. This notion is closely related to the notion of a collapsible $d$-hypertree (see \cite{linial19} and Section~\ref{sec:prelim:structure}).

Throughout the rest of this paper, we refer to a $1$-tree as a \emph{loose tree} (also known in the literature as a \emph{linear tree}). Similarly, a $(k-1)$-tree is also known as a \emph{tight tree}. The \emph{minimum $\ell$-degree} $\delta_{\ell}(H)$ of a $k$-uniform hypergraph $H$ is the minimum number of edges any set of $\ell$ vertices is contained in. Minimum $(k-1)$-degree is also known as \emph{codegree}, and minimum $1$-degree is known as \emph{vertex degree}. Maximum degree is defined analogously. Not much is known about extensions of Koml\'os, S\'ark\"ozy and Szemer\'edi's result to general $k$-uniform $\ell$-trees, apart from a recent result of Pavez-Sign{\'e}, Sanhueza-Matamala and Stein \cite{pavez2024dirac} which shows that a minimum codegree of $(1/2+\gamma)n$ forces the existence of any tight spanning tree of bounded vertex degree.

In the current work, we investigate minimum vertex degree conditions that guarantee the existence of loose spanning trees in $3$-uniform hypergraphs (also referred to as $3$-graphs from now on).
A hypergraph is a loose tree if and only if it is connected and has no Berge cycles (see Proposition~\ref{prop:berge}). This can be seen as an analogue of the acyclicity and connectedness conditions that define a $2$-uniform tree, making loose trees a natural generalization of trees in hypergraphs.

What minimum degree conditions in a $3$-graph force it to contain every loose spanning tree of bounded vertex degree? How are they related to thresholds for containment of Hamilton cycles and paths?

As in the graph case, a loose path is a special case of a bounded-degree loose tree. Bu\ss, H\`an and Schacht \cite{buss2013minimum} showed that if every vertex in an $n$-vertex $3$-graph is contained in at least $(7/16+\eps)\binom{n}{2}$ edges, then it contains a loose Hamilton cycle --- an $n$-vertex cycle whose adjacent edges share exactly one vertex. The constant $7/16$ is best possible, and in a later paper Han and Zhao \cite{han2015minimum} gave the exact threshold. A loose Hamilton path is not a subgraph of a loose Hamilton cycle, but slightly modifying the proof in~\cite{buss2013minimum} shows that $7/16$ is also the correct asymptotic threshold for loose Hamilton paths.

\begin{figure}
    \begin{center}
    \includegraphics[scale=0.9]{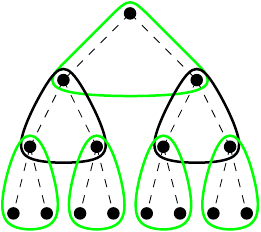}
    \end{center}
    \label{fig:binary-tree}
    \caption{The binary loose tree with $4$ levels. A perfect matching is shown in green. The dashed lines represent the edges of the simple binary tree from which the loose tree is obtained.}
\end{figure}

In light of this, one may conjecture that, like in the graph case, $3$-graphs with minimum vertex degree $(7/16+\eps)\binom{n}{2}$ also contain every loose tree of bounded degree. However, this is not the case. For any $b\ge 2$, consider the \emph{binary loose tree} $T_b$ whose edge set is the set of cherries consisting of a parent and two child vertices in a $b$-layer (simple) binary tree. An example for $b=4$ is shown in Figure~\ref{fig:binary-tree}. If the number of levels $b$ is even, $T_b$ contains a perfect matching, so any $3$-graph without a perfect matching will also not contain $T_b$. The asymptotic minimum degree threshold for perfect matchings in $3$-graphs was shown to be $5/9$ by H\`an, Person and Schacht \cite{han2009_pm}. Their asymptotic bound was later made exact by K\"uhn, Osthus and Treglown \cite{treglown2011full} and independently by Khan \cite{khan2013perfect}. This is tight as witnessed by the hypergraph on vertex set $A\cup B$ with $|A|=n/3-1$ and $|B|=2n/3+1$ consisting of all edges with at least one vertex in $A$. 

We show that this degree condition suffices for any spanning loose tree of bounded degree.
\begin{theorem}\label{thm:main}
For all $\gamma>0$ and $\Delta\in \N$ there exists $n_0\in\N$ such that any $3$-graph $H$ on $n\ge n_0$ vertices with $n$ odd and $\delta_1(H)\ge (5/9+\gamma)\binom{n}{2}$ contains every $n$-vertex loose tree $T$ with $\Delta_1(T)\le \Delta$.
\end{theorem}
This answers a question of Stein~\cite[Section 9]{stein2020tree}, and is the first result giving a minimum vertex degree condition for spanning hypertrees other than Hamilton paths. Again, our bound is asymptotically tight because a complete binary loose tree with an even number of levels contains a perfect matching and the $5/9$ threshold for perfect matchings is tight.

Note that every loose tree can be made tight by adding edges while keeping the vertex set unchanged and the maximum vertex degree bounded, so bounded-degree loose trees appear at codegree asymptotically $n/2$ by the main result in \cite{pavez2024dirac}. 
This is tight because there are $3$-graphs with minimum codegree $n/2-O(1)$ which do not contain a perfect matching, thus avoiding the binary loose tree $T_b$ --- for instance, consider the $3$-graph on vertex set $A \cup B$ with $|A|=|B|=n/2$ both odd where all edges with an even number of vertices in $B$ are present. However, whereas a minimum codegree condition implies some minimum degree condition (in fact, for any $k$-graph $G$ and $0<j<j'<k$ we have $\delta_{j'}(G)/\binom{n-j'}{k-j'}\le \delta_j(G)/\binom{n-j}{k-j}$, which implies that the threshold coefficients for any subgraph containment decrease with $j$ increasing), there are $3$-graphs with quadratic minimum degree in which some pairs of vertices have codegree $0$. For example, the $3$-graph on vertex set $A \cup B$ with edges all triples whose intersection with $B$ has size at least $2$, has minimum vertex degree $(1-o(1))\binom{n}{2}$ for $|A| \ll n$ but minimum codegree $0$ as long as $|A| \geq 2$. Thus, a minimum vertex degree condition constitutes a much weaker assumption on the underlying graph.\\

Our proof employs a classical recipe prescribed by the absorbing method. We give a high-level sketch of the main steps of the proof. 
\vspace{-10pt}\paragraph{Step 1.} Find a partial embedding of $T$ and a set $\cA$ of small absorbing substructures in $H$ with two key properties. First, the absorbing structures are flexible, that is, some vertices in them can be exchanged for other vertices of the host graph. Second, the partial embedding interacts appropriately with all elements of $\cA$. Both properties are illustrated in Figure~\ref{fig:absorbing-tuple}. We show how to find $\cA$ with these two properties in Lemmas~\ref{lemma:absorbing_set} and \ref{lemma:inhabiting_lemma}.
\vspace{-10pt}\paragraph{Step 2.} Extend this embedding to an embedding of a subtree $T'\subseteq T$ covering all but $\nu n$ vertices of $T$. For this we employ the hypergraph regularity method. Existing embedding results for spanning trees that use regularity tend to use a disconnected structure, such as a perfect matching, in the reduced graph $\cR$ of an $\eps$-regular partition $\{V_0,...,V_t\}$~\cite{kss1995,kss2001,pavez2024dirac}. However, they require certain connectivity properties of $\c R$ which are not guaranteed by our minimum vertex degree condition. We introduce a novel approach which allows us to instead utilize a connected structure in $\cR$. In our case this is a tight Hamilton cycle $\cC\subseteq \cR$. By regularity, finding the required embedding reduces to finding an assignment $a:V(T')\to [t]$ of the vertices of $T'$ to the clusters of the regularity partition such that every edge of $T'$ is mapped to an edge of $\cC$ and the number of vertices assigned to each $V_i$ is significantly smaller than $|V_i|$ (more precisely, at most $(1-\zeta) |V_i|$ for some $\zeta$ such that $\eps\ll\zeta\ll \nu$). See Lemma~\ref{lemma:almost_embedding} for this reduction. Finding such an assignment is the main difficulty in this step of the proof. To do this, we split $T'$ into pieces, assign these pieces to different edges of a perfect matching in $\cC$, and then ``travel'' along $\cC$ to connect the pieces to each other. We make sure to assign only a constant number of vertices during these connections. See Lemma~\ref{lemma:vertex_assignment} for the construction of this assignment.
\vspace{-10pt}\paragraph{Step 3.} Use $\cA$ to complete the embedding of $T$. We show this is possible in Lemma~\ref{lemma:absorbing_lemma}.

This paper is organised as follows. In Section~\ref{sec:prelim} we start with some preliminaries on loose trees, and the weak regularity lemma for hypergraphs. In Section~\ref{sec:almost} we show how to embed all but an arbitrarily small proportion of a bounded-degree loose tree in a graph with minimum vertex degree $(5/9+\gamma)\binom{n}{2}$. In Section~\ref{sec:absorb} we give an absorption strategy for completing the embedding from Section~\ref{sec:almost}. In Section~\ref{sec:proof}, we put these two ingredients together and prove Theorem~\ref{thm:main}. Finally, in Section~\ref{sec:concluding} we give some concluding remarks and open problems.

\section{Preliminaries}\label{sec:prelim}
\paragraph{Notation.} Given a $3$-graph $G$, we write $\deg_G(v;A,B)$ for the number of pairs $a\in A, b\in B$ such that $\{v,a,b\}\in E(G)$. Given three disjoint sets $A,B,C$, we write $G[A,B,C]$ for the tripartite subgraph of $G$ consisting of all edges with one vertex in each of $A,B$ and $C$ and $d_G(A,B,C)$ for the number of edges of $G[A,B,C]$ normalised by $|A||B||C|$. In particular, $\deg_G(v;A,B)=d_G(\{v\},A,B)|A||B|$. Usually the host graph is clear from the context and we omit the subscript. We sometimes refer to loose $3$-trees as \emph{loose trees} or just \emph{trees}.

We call an injective mapping $\phi:V(H) \rightarrow V(G)$ \emph{an embedding} of $H$ into $G$ if for any $\{x,y,z\} \in E(H)$, we have that $\{\phi(x), \phi(y), \phi(z)\} \in E(G)$. We denote by $\im(\phi) = \{\phi(x) | x\in V(H) \}$ the image of $\phi$. For some $S \subseteq V(H)$, we also denote $\phi(S) = \{\phi(v) | v \in S\}$.

We use standard notation for the hierarchy of positive constants: we say that a statement holds for $\alpha\ll\beta$ if there exists a non-decreasing function $f$ such that the statement holds if $\alpha<f(\beta)$.

Throughout the following sections, we often index collections of sets cyclically. To describe this, we use the function $m\mapsto m \mod n$ which takes an integer $m$ as input and returns the remainder of $m$ when divided by $n$, which is an integer in $\{0,1,...,n-1\}$.

\subsection{Tight Hamiltonicity in 3-graphs}
Since our proof relies on a certain reduced graph having a tight Hamilton cycle, we make use of the following result. A $3$-uniform tight cycle on vertex set $[n]$  has edges of the form $\{i,i+1,i+2\}$ taken modulo $n$ for all $0\le i\le n-1$.
\begin{theorem}[\cite{reiher2019minimum}]
\label{thm:degree-condition-tight-Hamiltonicity}
For every $\gamma >0$, there exists an $n_0 \in \mathbb{N}$ such that every $3$-graph $G$ on $n > n_0$ vertices with minimum degree $\delta_1(G) \geq (\frac{5}{9}+\gamma) \binom{n}{2}$ contains a tight Hamilton cycle. 
\end{theorem}

\subsection{The structure of hypertrees}\label{sec:prelim:structure}

In this section we give some basic facts and definitions about trees in hypergraphs, which are helpful in developing an intuition for their structure.

Recall that a $k$-uniform $\ell$-tree is a $k$-uniform hypergraph whose edges $e_1,...,e_m$ can be ordered so that each $e_i$ consists of $\ell$ vertices from a previous edge in the ordering and $k-\ell$ new vertices. We call the $\ell$ pre-existing vertices \emph{parents} and the $k-\ell$ new vertices \emph{children}. 

\begin{definition}
A \emph{layering} of a $3$-uniform rooted loose tree $T$ with root $r$ is a partition of its vertex set into layers $L_1, \dots, L_{\ell}$ such that $L_1 = \{r\}$ and for each $e = \{u,v,w\} \in E(T)$ such that $u$ is the parent of $v$ and $w$ in $T$, there is some $i\in [\ell-2]$ such that $u \in L_i$, $v \in L_{i+1}$, and $w \in L_{i+2}$.
\end{definition}
Note that each rooted loose tree has at least one layering, which can be obtained by setting $L_1 = \{r\}$ and then repeatedly considering any edge $e = \{u,v,w\}$ of the tree such that $u\in L_i$, and $v$ and $w$ have not yet been assigned a layer, and adding $v$ to $L_{i+1}$ and $w$ to $L_{i+2}$. The choice of assigning $v$ to $L_{i+1}$ and $w$ to $L_{i+2}$ instead of vice versa is arbitrary, which is why as soon as $T$ has at least one edge, it has more than one layering. For instance, if $E(T) = \{\{r, x, y\}\}$, then $T$ has two layerings --- $L_1 = \{r\}, L_2 = \{x\}, L_3 = \{y\}$ and $L_1 = \{r\}, L_2 = \{y\}, L_3 = \{x\}$.

Other than the layering, which is a commonly used property of $2$-uniform trees, loose trees can be characterised by their lack of cycles. The appropriate notion of a cycle here is the \emph{Berge cycle}, which is a cyclic sequence of alternating distinct vertices and edges $e_1, v_1, \dots, e_k, v_k, e_{k+1}=e_1$ such that $v_i \in e_{i} \cap e_{i+1}$ for every $i\in[k]$.

\begin{proposition}\label{prop:berge}A connected $3$-uniform hypergraph $H$ is a loose tree if and only if it contains no Berge cycles.
\end{proposition}
\begin{proof}
    The forward implication follows from the fact that the edges of a Berge cycle cannot be ordered in a valid way.

    For the backward implication, suppose that $H$ contains no Berge cycle. First note that a pair of edges $e,e'$ sharing exactly two vertices $v,w$ form a Berge cycle given by the sequence $eve'we$. So $H$ is linear, meaning every two edges intersect in at most one vertex. Take a maximal valid ordering $e_1,...,e_k$ of edges of $H$. Note that these edges form a connected subgraph of $H$. If $k=e(H)$, we have that $H$ is a loose tree. Otherwise, since the graph $H$ is also connected and linear, for some $e_{k+1} \in E(H)\setminus\{e_1,\dots,e_k\}$, we have that $e_{k+1}\cap e_j=\{v\}$ for some vertex $v$. Since $e_{k+1}$ cannot be appended to the valid ordering, there must exist another edge $e_{j'}$ such that $e_{k+1}\cap e_{j'}=\{w\}$ for some vertex $w\neq v$. The subgraph $\{e_1,...,e_k\}$ is connected so there is a loose path $P=v\dots w$ in it, which is also a Berge path. Then the sequence $e_{k+1}Pe_{k+1} = e_{k+1} v \dots w e_{k+1}$ is a Berge cycle, which leads to a contradiction.
\end{proof}

\subsection{Simplicial hypertrees}
In this section, we give an example of a simplicial $2$-hypertree which cannot be guaranteed by any minimum degree or codegree condition. Recall that a simplicial $d$-hypertree $T$ is a $d$-dimensional simplicial complex with a full $(d-1)$-skeleton such that $H_d(T;\mathbb{Q})=H_{d-1}(T;\mathbb{Q})=0$. The former means that $T$ consists of a vertex set $[n]$ (the $0$-faces), all $\binom{n}{\le d-1}$ sets of at most $d-1$ vertices (the $1$-faces through to the $(d-2)$-faces), and some sets of $d$ vertices (the $(d-1)$-faces), which are the hyperedges of our tree. To ensure the latter property we restrict our attention to the special class of \emph{collapsible hypertrees}. A $(d-1)$-face $\tau$ in a $d$-complex is called \emph{exposed} if there is exactly one $d$-face $\sigma$ that contains it. Removing this exposed $(d-1)$-face consists of removing both $\sigma$ and $\tau$. We say that a $d$-complex is \emph{collapsible} if we can iteratively remove exposed $(d-1)$-faces until no faces are left. As mentioned in~\cite{linial19}, if a $d$-dimensional complex is collapsible and has $\binom{n-1}{d}$ $d$-faces, then it is a simplicial $d$-hypertree. We use this fact to give a combinatorial construction of a simplicial $2$-hypertree and an almost-complete $3$-uniform hypergraph which doesn't contain it.

\begin{proposition}\label{prop:3ap_example}
    Let $1/n\ll \delta$ with $n$ even and let $\cA$ be the $3$-uniform hypergraph with vertex set $V=\{1,...,n\}$ and edge set $\{\{i,\lfloor \frac{i+j}{2} \rfloor,j\}:i,j\in [n], j-i \geq 2\}\}$. Let $\c H$ be the $3$-uniform hypergraph on the same vertex set with edge set $E(\c H)=\binom{V}{3}\setminus \binom{[\delta n]}{3}$. Then
    \begin{enumerate}
        \item $\c A$ is a simplicial $2$-hypertree with vertex degrees $O(n)$ and codegrees $O(1)$,
        \item $\c H$ has minimum degree $(1-\delta^2 - o(1))\binom{n}{2}$ and minimum codegree $(1-\delta)n$, and
        \item $\c A\not\subseteq \c H$.
    \end{enumerate}
\end{proposition}
\begin{proof}
    Firstly, the definition of $E(\cA)$ directly shows that all vertex degrees are of order $n$ and all codegrees are constant. It suffices to show that $\c A$ is collapsible and that it has $\binom{n-1}{2}$ edges. For the latter, note that for each pair $i,j \in [n]$ with $j-i \geq 2$, there is precisely one edge $\{i,\lfloor \frac{i+j}{2} \rfloor,j\}$ in $E(\cA)$. Thus, the number of edges in $\cA$ is the same as the number of pairs $i,j \in [n]$ with $j-i \geq 2$, of which there are 
    $$ \binom{n}{2} - (n-1) = \frac{(n-2)(n-1)}{2} = \binom{n-1}{2},$$
    since there are $n-1$ pairs $i,j \in [n]$ with $j-i=1$.

    To show that $\cA$ is collapsible, it is enough to exhibit an ordering of the edges $e_1, \dots, e_m$ with $m = \binom{n-1}{2}$, such that for each $i \in [m]$, some pair of vertices in the edge $e_i$ is exposed in the hypergraph $\{e_1, \dots, e_i\}$. The ordering $e_1, \dots, e_m$ is obtained by sorting the edges increasingly by the difference between their two most extreme vertices, breaking ties arbitrarily. That is, if $e=\{i, \lfloor \frac{i+j}{2} \rfloor,j\}$ and $f = \{k, \lfloor \frac{k+\ell}{2} \rfloor, \ell\}$, then if $|j-i| < |\ell - k|$, $e$ would be before $f$ in the order $e_1, \dots, e_m$. Now for each $i\in[m]$, the edge $e_i = \{k, \lfloor \frac{k+\ell}{2} \rfloor, \ell\}$ contains the pair $\{k,\ell\}$ which is exposed in $\{e_1, \dots, e_i\}$. This is because every edge $e_h = \{k, \ell, j\}$ different from $e_i$ must satisfy either $\ell = \lfloor \frac{k+j}{2} \rfloor$ or $k = \lfloor \frac{\ell+j}{2} \rfloor$. In the former case we have that $|k-j| > |\ell - k|$, in the latter case --- that $|\ell-j| > |\ell - k|$, so in both cases $e_h$ must come after $e_i$ in our order, that is, $h > i$. Thus, the pair $\{k, \ell\}$ does not appear in any of $e_1, \dots, e_{i-1}$, so it is exposed in $\{e_1, \dots, e_i\}$.
  
    Now we turn to $\cH$. Vertices $v\in V$ such that $v>\delta n$ have full degree. Vertices $v\le \delta n$ have degree at least $\binom{n-1}{2}-\binom{\delta n-1}{2}=(1-\delta^2 - o(1))\binom{n}{2}$. Similarly, pairs of vertices with at least one vertex outside of $\binom{[\delta n]}{2}$ have full codegree, and pairs of vertices in $\binom{[\delta n]}{2}$ have codegree at least $n-\delta n$.
    
    Finally, suppose for contradiction that $\c A\subseteq \c H$. Take $n$ to be large enough so that by the $k=3$ case of Szemer\'edi's theorem~\cite{szemeredi1975sets} (known also as Roth's theorem~\cite{roth1953certain}) every subset of $[n]$ of size at least $\delta n$ contains a $3$-term arithmetic progression. Then, whatever the embedding of $\c A$ into $\c H$ is, there will be an edge of $\c A$ entirely contained within $[\delta n]$, which is empty in $\c H$, leading to a contradiction.
\end{proof}

\subsection{The hypergraph regularity method}
In this subsection, we introduce the weak regularity lemma for hypergraphs, a generalization of Szemer\'edi's regularity lemma~\cite{szemeredi1975regular}.
Given $\eps > 0$, we say that a tripartite $3$-graph $G$ on sets $V_1, V_2, V_3$ is
\emph{$\eps$-regular} if for every $X_i \subseteq V_i$ of size at least $\eps|V_i|$, we
have
\[
  \left|d_G(X_1,X_2,X_3) - d_G(V_1,V_2,V_3)\right| \le \eps.
\]

A partition $\{V_0,...,V_t\}$
of the vertex set of an $n$-vertex $3$-graph $G$ is \emph{$\eps$-regular} if $|V_0| \le \varepsilon n$, all other sets $V_i$ with $i\in[t]$ have equal size, and the graph induced by all but at most
$\eps\binom{t}{3}$ triples $V_i,V_j,V_k$ with $i<j<k$ is $\eps$-regular. 

\begin{theorem}[Hypergraph regularity lemma~\cite{chung1991regularity,frankl1992uniformity}]
\label{thm:hypergraph-regularity}
  For every $\eps > 0$ and $t_0\in \N$, there exist $T_0,n_0\in \N$ such that every
  $3$-graph $G$ with at least $n_0$ vertices admits an $\eps$-regular
  partition $\{V_0,V_1,...,V_t\}$, where $t_0 \le t \le T_0$.
\end{theorem}

Given an $\eps$-regular partition $\cQ = \{V_0,...,V_t\}$ of the vertex set of a $3$-graph $G$,
we define the \emph{reduced graph} $\cR = \cR(\cQ, \eps, \alpha)$ on
vertex set $\{1,\dotsc,t\}$ corresponding to the sets $\{V_0, V_1,...,V_t\}$. The edges of $\cR$ are
all $\eps$-regular triples of density at least $\alpha$. A special case of the following by now standard degree inheritance lemma was first shown in~\cite[Proposition 15]{buss2013minimum}, and later generalized in~\cite[Corollary 2.5]{han2015minimum}.
\begin{lemma}[Hypergraph regularity lemma --- degree version]
\label{lemma:degree-inheritance}
For every $0 < \eps < \alpha < \delta$ and $t_0 \in \mathbb{N}$, there exist $T_0, n_0\in \N$ such that the following holds. Suppose that $G$ is a $3$-graph on $n \ge n_0$ vertices with minimum degree $\delta_1(H) \geq \delta \binom{n}{2}$. Then $G$ admits an $\eps$-regular partition $\cQ = \{V_0,...,V_t\}$ such that $t_0 < t < T_0$ and the reduced graph $\cR = \cR (\cQ, \eps, \alpha)$ has minimum degree $\delta_1(\cR) \geq (\delta - \eps - \alpha)\binom{t}{2}$.
\end{lemma}

Given disjoint vertex sets $X,Y$ in a $3$-graph, we say that a vertex $v\notin X \cup Y$ \emph{$d$-expands into $\{X,Y\}$} if $\deg(v;X,Y)\ge d|X||Y|$. Embedding trees in regular triples, we will aim to choose edges consisting of expanding vertices. The following lemma will be a key tool allowing us to maintain this property throughout the embedding.

\begin{lemma}\label{lem:typical}
Let $0<1/n_0 \ll \varepsilon \ll d$ and let $t\in\mathbb{N}$.
Let $G$ be a $3$-graph with vertex set $V_1 \cup \dots \cup V_t \cup \{x\}$, where the indices of $V_i$ behave cyclically and $\forall i \neq j, V_i \cap V_j = \varnothing$, but possibly $x \in V_i$ for some $i\in[t]$. Suppose each $|V_i|\geq n_0$ and that for all $i\in[t]$ the tripartite graph $G[V_i,V_{i+1},V_{i+2}]$ is $\eps$-regular of density at least $d$. Let $X_i, Y_i, Z_i\subseteq V_i$ for $i\in[t]$ (not necessarily distinct) each have size at least $\sqrt{\eps} |V_i|$, and suppose that $x$ is a vertex which $d/8$-expands into $\{X_k,X_\ell\}$, where $k,\ell\in [t], k \neq \ell$. Then there exist vertices $y\in X_k,z\in X_\ell$, such that $xyz$ is an edge, $y$ is $d/4$-expanding into $\{Y_{k-2},Y_{k-1}\}$,$\{Y_{k-1},Y_{k+1}\}$ and $\{Y_{k+1},Y_{k+2}\}$, and $z$ is $d/4$-expanding into $\{Z_{\ell-2},Z_{\ell-1}\}$,$\{Z_{\ell-1},Z_{\ell+1}\}$ and $\{Z_{\ell+1},Z_{\ell+2}\}$.
\end{lemma}
\begin{proof}
The vertex $x$ is $d/8$-expanding into $\{X_k,X_\ell\}$, so by averaging, $d|X_k|/16$ vertices $y\in X_k$ satisfy $\deg(x,y;X_\ell)\ge d|X_\ell|/16$. Indeed, denote by $A\subseteq X_k$ the set of such ``expanding'' vertices $y\in X_k$. If $|A|<d|X_k|/16$ we have 
\begin{align*}\deg(x;X_k,X_\ell)&= \sum_{y\in A}\underbrace{\deg(x,y;X_\ell)}_{\le |X_\ell|}+ \sum_{y\in X_k\setminus A}\underbrace{\deg(x,y;X_\ell)}_{\le d|X_\ell|/16}\\
&< (d|X_k|/16)\times |X_\ell|+|X_k|\times (d|X_\ell|/16)=d|X_k||X_\ell|/8,\end{align*}
a contradiction.

Since $|A|=d|X_k|/16\ge\eps|V_k|$, by regularity $d(Y_{k-2},Y_{k-1},A)\ge d/2$ and by an analogous averaging at least $d|A|/4$ of vertices $a\in A$ satisfy $\deg(a;Y_{k-2},Y_{k-1})\ge \frac{d}{4}|Y_{k-2}||Y_{k-1}|$, i.e. they are $d/4$-expanding into $\{Y_{k-2},Y_{k-1}\}$. Denote the set of such vertices by $A'$. Similarly, a $d/4$-fraction of the vertices in $A'$ are $d/4$-expanding into $\{Y_{k-1},Y_{k+1}\}$ and a $d/4$-fraction of those vertices are $d/4$-expanding into $\{Y_{k+1},Y_{k+2}\}$. This yields a set of at least $d^3|A|/64>0$ vertices in $A$ which are $d/4$-expanding into all three triples needed for $y$. Pick $y$ to be one of these vertices.

By construction, we have $\deg(x,y;X_\ell)\ge d|X_\ell|/16 \ge \eps|V_i|$. We can now repeat the above procedure to obtain a set of at least $d^4|X_\ell|/1024>0$ vertices in the co-neighbourhood of $x$ and $y$ which are $d/4$-expanding into all three triples needed for $z$. Take $z$ to be one of these vertices.
\end{proof}

\section{Covering almost all vertices}\label{sec:almost}

In this section we show that Theorem~\ref{thm:main} holds for almost-spanning trees. The main result of this section is Lemma~\ref{lemma:almost_embedding}, in which we embed an almost-spanning tree using a weak regularity partition of the host graph. We start with a helper lemma which tells us which cluster of the partition each vertex of the tree will be embedded in. This lemma already prescribes the global strategy of our embedding: we break our tree into pieces, assign each piece to an edge in the reduced graph, and then use the edges of the tight Hamilton cycle in the reduced graph to connect every pair of pieces that are adjacent (see Figure~\ref{fig:almostcover}). For technical reasons (since we want the freedom to later embed the root of our almost spanning tree wherever we want), we already work with up to $2\Delta$ disjoint subtrees in this helper lemma, which are then broken into smaller subpieces.

\begin{lemma}
\label{lemma:vertex_assignment}
Let $1/n' \ll 1/t \ll \zeta \ll \nu \ll 1/\Delta$ and suppose $t\equiv 1 \pmod 3$.
Let $T^1, \dots T^s$ be disjoint loose trees of maximum vertex degree $\Delta$ rooted at $r_1, \dots, r_s$, where $1 \le s \le 2 \Delta$ and $\sum_{i\in[s]} v(T^i) \le (1-\nu)n'$. Let $a':\{r_1,...,r_s\}\to [t-2]$ be an assignment of the roots to distinct integers in $[t-2]$. Then $a'$ can be extended to a full assignment $a:\bigcup_{i\in[s]}V(T^i)\to[t]$ such that 
\begin{enumerate}[label=(\roman*)]
    \item \label{cond:room-left-each-cluster} $|a^{-1}(j)|\le (1-\zeta)n'/t$ for all $j\in[t]$, and
    \item \label{cond:embedding-in-edge} for all $e\in \bigcup_{i\in[s]}E(T^i)$ we have $a(e)=\{j,(j\mod t) +1,(j+1\mod t)+1\}$ for some $j\in [t]$.
\end{enumerate}
\end{lemma}
\begin{proof}
Choose $\beta$ satisfying  $1/n'\ll \beta \ll 1/t$.
For each $i\in[s]$ and $v \in V(T^i)$, denote by $T^i(v)$ the subtree of $T^i$ rooted at $v$.

Firstly, we decompose each tree $T^i$ into a number of subtrees $T^i_1, \dots, T^i_{k_i}$ in the following way. Start with $T' := T^i$. Find a vertex $v \in T'$ such that the subtree $T^i(v)$ has at least $\beta n'$ vertices, but for each $u$ among $v$'s children, the subtree $T^i(u)$ has fewer than $\beta n'$ vertices. Such a vertex can be found by starting from the root and iteratively moving to the child of the current vertex that has the largest subtree, until reaching a vertex for which the condition holds. Set $T^i_1 := T^i(v)$ and modify $T'$ by removing $V(T^i(v))\setminus\{v\}$ (that is, replacing $T^i(v)$ with a leaf). Continue iteratively, constructing $T^i_2, T^i_3, \dots$ in the same way, until at some point the tree $T'$ has at most $2\Delta \beta n'$ vertices, then add $T'$ as the last tree $T^i_{k_i}$ in the decomposition, noting that $r_i\in V(T^i_{k_i})$. All these subtrees have size at most $2 \Delta \beta n'$ and all but perhaps $T^i_{k_i}$ have size at least $\beta n'$. Letting $k := \sum_{i=1}^s k_i$, we rename all $k$ trees we have obtained from our decomposition as $T_1, \dots, T_k$ in the following specific order. For $i\in[s],$ we let $T_i := T^i_{k_i}$ and for every $j\in [k_i - 1]$, we set $h := s + \sum_{i'=1}^{i-1} (k_{i'}-1) + (k_i-j)$ and let $T_h := T^i_j$. That is, the first $s$ trees in our order are $T^1_{k_1}, \dots, T^s_{k_s}$, and after that we iterate for $i$ from $1$ to $s$ and place all the subtrees of $T^i$ in reverse order (i.e. $T^i_{k_i-1}, \dots, T^i_1$). Crucially, for each $T_i$ with $i>s$ in this order, the root of $T_i$ is a leaf in some $T_j$ with $j < i$. Note that $k\le 1/\beta + s \leq 1/\beta + 2\Delta$. 

For each $z\in \{1,\dots,\lfloor t/3\rfloor\}$, let $E_z = \{3z-2, 3z-1, 3z\}$ and denote by $M = \{E_1, \dots, E_{\lfloor t/3\rfloor}\}$ the matching formed by the disjoint consecutive triples $E_z$. Note that only the index $t$ does not belong to any edge in this matching.

For each $T_i$, pick an arbitrary layering $L_1, L_2, \dots$ and colour its vertices in 3 colours $1,2,3$ according to the layering, that is, colour layer $L_i$ in colour $(i-1) \mod 3 + 1$. This yields a proper colouring of $T_i$, that is, each edge of $T_i$ contains precisely one vertex of each colour. Let $C_1(T_i), C_2(T_i), C_3(T_i)$ be the colour classes that correspond to each colour.

We will extend $a'$ to the vertices of each of $T_1, T_2, \dots$ in order, and for each $T_i$ most of the vertices of $T_i$ will be assigned to the three indices in a single edge $E_z\in M$. We will use the consecutive triples of $[t]$ (taken cyclically) to `travel' from the edge each subtree is assigned to the edge its child subtree is assigned to.

At first, we set $a := a'$ and we will gradually extend it.
For each $h\in[t]$, we define  $\ell(h):=n'/t-|a^{-1}(h)|$. We refer to this as the \emph{leftover} of $h$ and note that to satisfy~\ref{cond:room-left-each-cluster}, at the end of the embedding we must have $\ell(h)\ge \zeta n'/t$ for all $h$. Whenever we extend the assignment by setting $a(v)=h$, we decrease $\ell(h)$ by $1$. Since $\sum_{i\in[s]}v(T^i)\le (1-\nu)n'$, at any stage of the embedding all leftovers must sum to at least $\nu n'$.

Throughout, we maintain the following \emph{balance invariant}: for each $z\in\{1,\dots, \lfloor t/3\rfloor\}$, after assigning trees $T_1, \dots, T_i$ to clusters, the differences between $\ell(3z-2), \ell(3z-1)$ and $\ell(3z)$ are all at most $6\Delta \beta n' + i\times 6t(2\Delta)^{6t}$.

For each $j\in[s]$, assign the colour classes of $T^j_{k_j}$ to $a'(r_j)$ and the other two elements in $E_{\lceil a'(r_j)/3\rceil}$ in such a way that~\ref{cond:embedding-in-edge} is satisfied and $r_j$ is assigned to $a'(r_j)$, as required. Note that since each $T_i$ has size at most $2\Delta\beta n'$, and all $a'(r_j)$ are distinct, this satisfies the balance invariant in a strong sense --- that is, after assigning all trees $T_{k_j}^j$ with $j\in[s]$, we have that for every $z$ the differences between $\ell(3z-2), \ell(3z-1)$ and $\ell(3z)$ are all at most $6\Delta \beta n'$.

Suppose we have assigned trees $T_1, T_2, \dots, T_{i-1}$ and we now consider the unassigned $T_i$. We first find an edge $E_z$ such that $\min\{\ell(3z-2),\ell(3z-1), \ell(3z)\}\ge 2 \zeta n'/t$. Such an edge exists, for otherwise the total leftover is less than
$$ \ell(t) + \lfloor t/3\rfloor(6 \zeta n'/t + 12 \Delta \beta n'+2k\times 6t(2\Delta)^{6t}) \le \nu n'.$$

We will assign most of $T_i$ to $E_z$. Let $\pi$ be the permutation of $[3]$ such that $|C_{\pi(1)}(T_i)| \le |C_{\pi(2)}(T_i)| \le |C_{\pi(3)}(T_i)|$. Let $E_z = \{w_1, w_2, w_3\}$ with $w_1 = 3z-2, w_2 =3z-1, w_3 = 3z$, and let $\rho$ be the permutation such that $\ell(w_{\rho(1)}) \le \ell(w_{\rho(2)}) \le \ell(w_{\rho(3)})$ . Then for a balanced assignment we aim to assign most of $C_{\eta(h)}(T_i)$ to $w_{h}$, for $h \in [3]$, where $\eta=\pi\circ\rho^{-1}$.

Let $r$ be the root of $T_i$, and note that $a(r)=j$ is already defined by our choice of the order $T_1, \dots, T_k$. Let $L_1, L_2,...$ be the layering of $T_i$ that we chose. If $j = w_h$ with $\eta(h)=1$, then $r$ is where it should be, so we simply assign all of $C_{\eta(h)}(T_i)$ to $w_h$, for each $h \in [3]$.

Otherwise, we start assigning $T_i$ to integers in $[t]$ layer by layer, traversing the cyclic ordering of $[t]$ possibly several times, in order to get to $E_z$ and assign most of each color class to its intended index. We already know that $L_1=\{r\}$ is assigned to $j$. Suppose layer $L_b$ was assigned to $h\in [t]$. We consider two cases. If $(\eta(1), \eta(2), \eta(3)) \in \{ (1,2,3), (2,3,1), (3,1,2)\}$, then we assign layer $L_{b+1}$ to $(h \mod t) + 1$. We stop as soon as we assign some layer $L_{b'}\subseteq C_{\eta(g)}(T_i)$ to $w_g$. Note that since $\eta$ is cyclic, and by the sequential nature of the embedding, we must have $g=1$. This means $L_{b'}$ has now been assigned to its intended index. Note that this will happen for some $b' \le 6t$ because for each $h\in[t]$, if $L_w$ is some layer assigned to $h$, then the next two layers assigned to $h$ will be $L_{w+t}$ and $L_{w+2t}$ and since $t$ is not divisible by $3$, these three layers $L_w, L_{w+t}, L_{w+2t}$ will have different colours. We complete the assignment of $T_i$ by assigning each subsequent layer $L_{b'+j}$ to $w_{(j \mod 3)+1}$.

Otherwise, we have $(\eta(1), \eta(2), \eta(3)) \in \{ (1,3,2), (2,1,3), (3,2,1)\}$ and we assign layer $L_{b+1}$ to $(h - 2\mod t)+1$ instead, that is, we traverse $[t]$ in the opposite direction. Again, we stop when we assign some layer $L_{b'}\subseteq C_{\eta(g)}(T_i)$ to $w_g$, and this happens for $b'\le 6t$. Analogously to the previous case, we have $g=3$. We complete the assignment of $T_i$ by assigning all subsequent layers $L_{b'+j}$ to $w_{3-(j\mod 3)}$.

In both cases, we assign all edges of $T_i$ to triples of consecutive elements in the natural cyclic ordering of $[t]$. Note that only constantly many vertices of $T_i$ are assigned to indices outside of $E_z$ and all other vertices are assigned to $E_z$. This is because $T_i$ has maximum degree at most $\Delta$, so the number of vertices at layers $1, \dots, 6t$ is at most $6t(2\Delta)^{6t}$.

From the balance invariant, before defining the assignment for $T_i$, we had $0\le\ell(w_{\rho(3)})-\ell(w_{\rho(1)})\le 6\Delta \beta n'+(i-1)\times 6t(2\Delta)^{6t}$. After assigning the vertices of $T_i$ we have that the new leftovers still satisfy the balance invariant since
\begin{align*}-2\Delta \beta n' - 6t(2\Delta)^{6t} &\le\ell_{new}(w_{\rho(3)})-\ell_{new}(w_{\rho(1)})\\
&\le\ell(w_{\rho(3)})-\ell(w_{\rho(1)})+\underbrace{|C_{\pi(1)}(T_i)|-|C_{\pi(3)}(T_i)|}_{\le 0}+6t(2\Delta)^{6t}\\
&\le 6\Delta\beta n' + i\times 6t(2\Delta)^{6t},
\end{align*}
and similarly for $\{w_{\rho(1)}, w_{\rho(2)}\}$ and $\{w_{\rho(2)}, w_{\rho(3)}\}$. For any $E_{z'}$ with $z' \neq z$, the difference between the leftovers of any two elements of $E_{z'}$ increases by at most $6t(2\Delta)^{6t}$.

We thus assign the vertices of all trees $T_1, T_2, \dots T_k$ to $[t]$ in a way that satisfies (ii). When assigning each $T_i$, we subtracted at most $|T_i|\le 2\Delta \beta n'$ from each leftover in $E_z$, chosen to be at least $2\zeta n'/t$. We also removed at most $6t(2\Delta)^{6t}$ from leftovers in $[t]\setminus E_z$. Thus at the end of the assignment $\ell(h) \ge 2\zeta n'/t-2\Delta\beta n' - k \times 6t(2\Delta)^{6t} \ge \zeta n'/t$ for all $h\in[t]$, showing that~\ref{cond:room-left-each-cluster} is satisfied.
\end{proof}

\begin{figure}
    \begin{center}
    \includegraphics[scale=0.5]{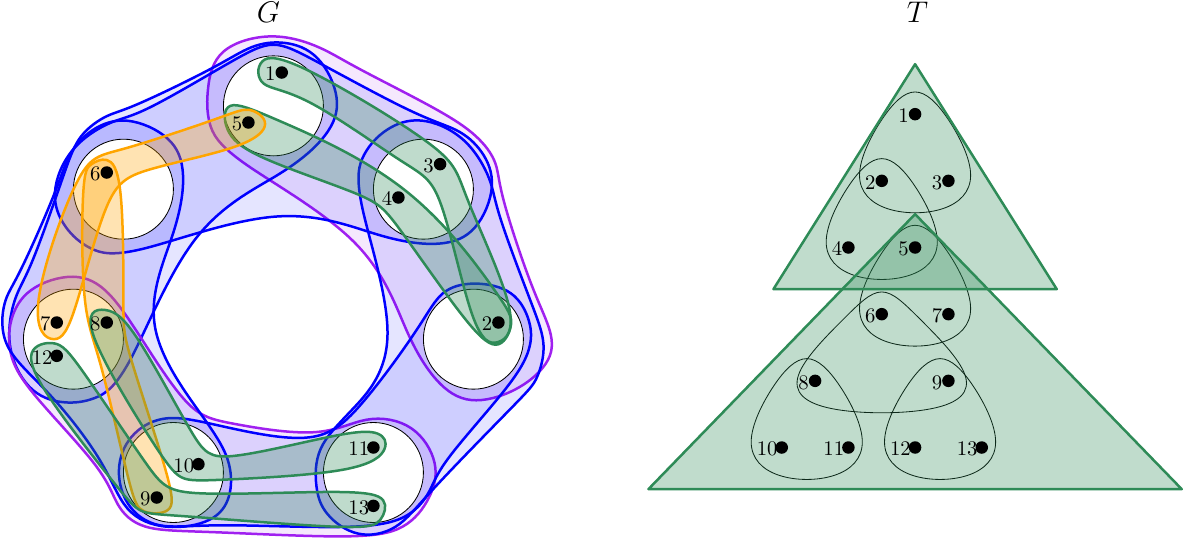}
    \end{center}
    \caption{The tree $T$ is embedded piece by piece (pieces shown in green), using the tight Hamilton cycle in the reduced graph to connect the pieces. The matching edges of the reduced graph in $G$ are highlighted in purple. The images of edges of $T$ that are used for ``connecting" the two parts of the tree along the tight Hamilton cycle are shown in orange. The images of the other edges of $T$, i.e., those that are embedded in matching edges, are shown in green. }
    \label{fig:almostcover}
\end{figure}

We now state and prove the main lemma of this section, showing that every almost-spanning tree can be embedded into our host graph, even with a prescribed embedding of its root.
\begin{lemma}(An almost embedding)
\label{lemma:almost_embedding}
Let $1/n \ll \nu \ll \gamma, 1/\Delta$, and let $G$ be a $3$-graph on $n$ vertices with $\delta_1(G) \geq (\frac{5}{9} + \gamma)\binom{n}{2}$. Let $T$ be a loose tree of maximum degree $\Delta$ on $(1-\nu) n$ vertices. Then for every $r \in V(T)$ and $x \in V(G)$, there exists an embedding of $T$ into $G$ that maps $r$ to $x$.
\end{lemma}
\begin{proof}
Choose $t_0, \eps$ and $\zeta$ such that $1/n \ll 1/t_0 \ll \varepsilon \ll \zeta \ll \nu$.
By Lemma~\ref{lemma:degree-inheritance}, there exists an $\eps$-regular partition $\cQ = \{V_0,...,V_t\}$ with $t_0 < t < T_0$ such that the reduced graph $\cR = \cR(\cQ, \eps, \gamma/2)$ has minimum degree $\delta_1(\cR) \geq (\frac{5}{9} + \gamma/3)\binom{t}{2}$. If $t\not \equiv 1 \pmod 3$ we set $V_0 := V_0 \cup V_t$ or $V_0:=V_0\cup V_t\cup V_{t-1}$ and decrease $t$ by $1$ or $2$ to ensure $t\equiv 1 \pmod 3$. The modified reduced graph $\cR$ still has minimum degree $\delta_1(\cR) \geq (\frac{5}{9} + \gamma/4)\binom{t}{2}$ since $t_0$ is large. By Theorem~\ref{thm:degree-condition-tight-Hamiltonicity}, the reduced graph $\cR$ contains a tight Hamilton cycle $\cC$. Recall that every edge of $\cC$ corresponds to an $\eps$-regular triple with density at least $\gamma/2$. Rename the vertices of $\cR$ such that $V_1, \dots, V_t$ gives the correct order for the tight Hamilton cycle.

We build an embedding $\phi:T \rightarrow G$ vertex by vertex, starting by defining $\phi(r) = x$. For each $i\in[t]$, let $R_i\subseteq V_i$ be a set of size $2\sqrt{\eps} |V_i|$ (the sets $R_i$ are reservoirs of vertices). At each step of the embedding, we will keep track of the sets of available \emph{reservoir} vertices $V^r_i:=R_i\setminus \im(\phi)$ and available \emph{main} vertices $V^m_i:=(V_i\setminus R_i)\setminus \im(\phi)$. Note that these sets get updated as we extend $\phi$, but we won't describe this explicitly to avoid cumbersome notation.

We would like to embed the rest of $T$ using vertices in $V_1, \dots, V_t$ which expand into the regular triples of $\cC$, but the fixed vertex $x$ could fall in $V_0$, or it could have low degree in the triples of $\cC$ containing it. To get around this, we will embed the first layer of $T$ manually. Let $\{r,r_1,r_2\},\{r,r_3,r_4\},\dots, \{r,r_{s-1}, r_s\}$ be the edges in $T$ that contain $r$, and let $T_1, \dots, T_s$ be the respective subtrees in $T$ rooted at $r_1, \dots, r_s$. Note that $s \leq 2\Delta$, since $r$ has degree at most $\Delta$.

\begin{claim}
There are distinct sets $V_{i_1}, \dots, V_{i_s}$ with $1 \le i_1 < \dots i_s \le t-2$ such that for all $j \in [s/2]$, we have that $\deg_G(x; V_{i_{2j-1}}, V_{i_{2j}}) \geq (5/9 + \gamma/2) |V_{i_{2j-1}}| |V_{i_{2j}}|$.
\end{claim}
\begin{proof}[Proof of claim.]
Take a maximum such set of distinct sets $V_{i_1}, \dots, V_{i_{s'}}$, and suppose that $s' < s$. This implies that the degree of $x$ in $G$ is at most
$$ |V_0 \cup V_{t-1} \cup V_t| n + |\bigcup_{j \in [s']} V_{i_j}|n + \sum_{i \in [t]} \binom{|V_i|}{2} + \sum_{i, j \in [t] \setminus \{i_1, \dots, i_{s'}\}, i< j} (5/9 + \gamma/2) |V_i| |V_j| $$
$$ \le \nu n^2 + \frac{2\Delta }{t} n^2 + t \binom{\frac{n}{t}}{2} + \binom{t}{2} (5/9 + \gamma/2) \frac{n^2}{t^2} < (5/9 + 3\gamma/4) \binom{n}{2},$$
contradicting the minimum degree condition $\delta_1(G) \geq (5/9 + \gamma) \binom{n}{2}$. 
\end{proof}
This in turn implies that for every $j\in[s/2]$, we have $$\deg_G(x; V^m_{i_{2j-1}}, V^m_{i_{2j}}) \geq \deg_G(x; V_{i_{2j-1}}, V_{i_{2j}}) - 4\sqrt{\eps}|V_{i_{2j-1}}||V_{i_{2j}}| \geq (5/9 + \gamma/4) |V^m_{i_{2j-1}}| |V^m_{i_{2j}}|.$$

For each $j \in [s/2]$, we set $\phi(r_{2j}) = v_{2j}$ and $\phi(r_{2j-1}) = v_{2j-1}$ where $\{x, v_{2j-1}, v_{2j}\}\in E(G)$ and for $h\in\{2j-1,2j\}$, vertex $v_h\in V^m_{i_h}$ is a $\gamma/8$-expanding vertex into $\{V^r_{i_h-2},V^r_{i_h-1}\}$, $\{V^r_{i_h-1},V^r_{i_h+1}\}$, $\{V^r_{i_h+1}, V^r_{i_h+2}\}$. This is possible by Lemma~\ref{lem:typical} since, as we just showed, $x$ has good degree to all pairs $V^m_{i_{2j-1}}, V^m_{i_{2j}}$.

Next, we apply Lemma~\ref{lemma:vertex_assignment} to assign vertices to sets $V_1, \dots, V_t$, where they will later be embedded. We do so with parameters $n'=n-|V_0|\ge(1-\eps)n$, $t:=t$, $\zeta:=\zeta$, $\nu:=\nu/2$, $\Delta:=\Delta$, trees $T_1, \dots, T_s$ rooted at $r_1, \dots, r_s$, and $a'(r_j)=i_j$ for $j\in[s]$ as specified above. We have that $\sum_{i\in[s]}v(T_i)\le (1-\nu)n\le (1-\nu/2)(1-\eps)n\le (1-\nu/2)n'$ and recall that $s \leq 2\Delta$, so all assumptions of Lemma~\ref{lemma:vertex_assignment} apply. Let $a: \bigcup_{i\in[s]}V(T_i) \to [t]$ be the resulting assignment of the vertices. With this assignment in mind, we aim to embed each $v$ that is not already in $\im \phi$ to $V_{a(v)}$. Towards this aim, we refer to $v$ as ``assigned to $V_{a(v)}$''. Note that for any edge $\{u,v,w\} \in E(T)$, apart from possibly the ones that $x$ is in, by property~\ref{cond:embedding-in-edge} of Lemma~\ref{lemma:vertex_assignment} we have that $V_{a(u)}, V_{a(v)}, V_{a(w)}$ is an edge in the tight Hamilton cycle $\cC$ in $\cR$. By property~\ref{cond:room-left-each-cluster}, for each $j \in [t]$ the number of vertices assigned to $V_j$ is at most $(1-\zeta)|V_j|$.

We start by splitting each tree $T_j$ into pieces of size between $\beta n$ and $2\Delta \beta n$ for $\beta\ll 1/t$ as we did in the beginning of the proof of Lemma~\ref{lemma:vertex_assignment}. Define a set of \emph{special vertices} consisting of the roots of all pieces. Further, take a valid ordering of the edges of each piece, and concatenate them in a breadth-first manner to obtain a valid ordering of each edge set $E(T_j)$. Concatenate all orderings of $E(T_j)$ in any order. The resulting ordering of edges has the important property that any two edges which intersect in a non-special vertex are at a distance at most $2\Delta \beta n$ in the ordering.

We proceed to define the rest of the embedding $\phi$ one edge at a time, according to the edge ordering given above. The children of special vertices will be embedded in the reservoir sets $V^r_\star$, and all other vertices in their complements $V^m_\star$ within the clusters of the regular partition. To keep track of this, for a vertex $v\in V(T)$, we define $\sigma(v)=r$ if $v$ is special and $m$ otherwise and $\tau(v)=r$ if $v$ is the child of a special vertex and $m$ otherwise. Intuitively, $\tau(v)$ says whether $v$ will be embedded in a reservoir, and $\sigma(v)$ says whether the children of $v$ will be embedded in a reservoir. Note that at most $1/\beta + 2\Delta$ vertices are mapped to $r$ by $\sigma$ and at most $2\Delta/\beta + 4\Delta^2$ vertices are mapped to $r$ by $\tau$. Throughout the embedding, we will maintain the following \emph{expansion property}: each vertex $v$ is mapped to a $\gamma/8$-expanding vertex into $\{V^{\sigma(v)}_{a(v)-2},V^{\sigma(v)}_{a(v)-1}\}$, $\{V^{\sigma(v)}_{a(v)-1}, V^{\sigma(v)}_{a(v)+1}\}$,$\{ V^{\sigma(v)}_{a(v)+1},V^{\sigma(v)}_{a(v)+2}\}$. Note that in our manual embedding we made sure that this property holds for $r_1,...,r_s$.

Suppose that we are looking to embed some edge $\{y, z, q\} \in E(T)$ such that $i=a(y), j = a(z), h=a(q)$ and $y' := \phi(y) \in V_i$ is already defined. Note that $i,j,h$ in some order are consecutive integers, where we identify $t$ with $0$ and $t-1$ with $-1$.

By the expansion property, at the time of embedding $y$, its image $y'$ was $\gamma/8$-expanding into $\{V^{\sigma(y)}_{i-2},V^{\sigma(y)}_{i-1}\}$, $\{V^{\sigma(y)}_{i-1}, V^{\sigma(y)}_{i+1}\}$,$\{ V^{\sigma(y)}_{i+1},V^{\sigma(y)}_{i+2}\}$. If $y$ is special, i.e. if $\sigma(y)=r$, since embedding $y$, we removed at most $3\Delta/\beta < \frac{\gamma}{64}|V^r_{\star}|$ vertices from $V^r_{i-2},...,V^r_{i+2}$. If $y$ is not special, i.e. if $\sigma(y)=m$, since embedding $y$, we removed at most $4\Delta \beta n< \frac{\gamma}{64} \times  \frac{\zeta}{2} |V_\star|<\frac{\gamma}{64}|V_\star^m|$ vertices from $V^m_{i-2},...,V^m_{i+2}$. In both cases, $y'$ remains $\gamma/16$-expanding. Now, by Lemma~\ref{lem:typical} (setting $d:=\gamma/2$) applied to the sets $V_1, \dots, V_t$ with $k =j, \ell = h$ and $x:=y'$, the $X$-sets being $V^{\sigma(y)}_g$, the $Y$-sets being $V^{\sigma(z)}_g$ and the $Z$-sets being $V^{\sigma(q)}_g$ for each $g\in [t]$, we obtain $z'\in V^{\sigma(y)}_j$ (noting that $\sigma(y)=\tau(z)$ because $z$ is a child of $y$) and $q'\in V^{\sigma(y)}_h$ (noting that $\sigma(y)=\tau(q)$ because $q$ is a child of $y$) which satisfy the expansion property and $\{y',z',q'\}\in E(G)$. Set $\phi(z)=z'$ and $\phi(q)=q'$ to complete the embedding of $\{y,z,q\}$.
\end{proof}

\section{Completing the embedding}\label{sec:absorb}

To finish the embedding, we use an absorption technique inspired by the one developed in~\cite{bottcher2020embedding} and used also in~\cite{bottcher2019universality,pavez2024dirac}. We start with some definitions.

\begin{definition}
A $d$-\emph{star} $S_v = \big( v, (u^i_1, u^i_2)_{i \in [d]} \big)$ in a $3$-uniform hypergraph consists of $d$ edges $\{v, u^i_1, u^i_2\}_{i \in [d]}$, where $v$ and $u^i_1, u^i_2$ for $i \in [d]$ are distinct. We refer to $v$ as the \emph{center} of $S_v$.
\end{definition}

\begin{definition}
Let $H$ be a $3$-graph.
For a triple $(w_1, w_2, w_3)$, a $d$-\emph{absorbing tuple} consists of $2$ vertex-disjoint $d$-stars $(S_{v_2}, S_{v_3})$ such that $\{w_1, v_2, v_3\} \in E(H)$ and for $i \in \{2,3\}$ with $S_{v_i} = \big(v_i, (u^{i,j}_1, u^{i,j}_2)_{j \in [d]} \big)$, we have that $S_{w_i} = \big(w_i, (u^{i,j}_1, u^{i,j}_2)_{j \in [d]} \big)$ is also a $d$-star in $H$. We denote by $A_d(w_1, w_2, w_3)$ the set of all $d$-absorbing tuples for $(w_1, w_2, w_3)$, and by $A_d(H)$ the union of all $A_d(w_1, w_2, w_3)$ for all triples $(w_1, w_2, w_3)\in V(H)^3$.

Let $T$ be a loose tree and $\phi: V(T) \rightarrow V(H)$ be an embedding of $T$ in $H$. We say that a star $S_v$ in $H$ is \emph{inhabited} by $\phi$ if there is $x \in V(T)$ such that $\phi(x) = v$ and each edge $e \in E(T)$ with $x \in e$ is mapped to a distinct edge of $S_v$. Note that some vertices and edges of $S_v$ may not belong to the image of $\phi$ and can thus be freshly used in extensions of $\phi$. We say that an absorbing tuple $(S_{v_2}, S_{v_3})$ is \emph{inhabited} by $\phi$ if both $S_{v_2}$ and $S_{v_3}$ are inhabited by $\phi$. These definitions are illustrated in Figure~\ref{fig:absorbing-tuple}.
\end{definition}

\begin{figure}
    \centering
\includegraphics{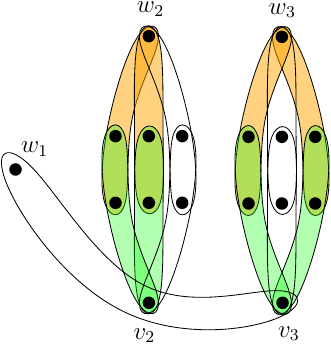}
    \caption{A $3$-absorbing tuple for $(w_1,w_2,w_3)$ inhabited by an embedding $\phi$. Images of edges under $\phi$ are shown in green. The crucial property of this structure is that the two green stars in $\phi$ can be switched for the two orange stars plus an extra edge at $w_1$, thus extending the embedding.}
    \label{fig:absorbing-tuple}
\end{figure}

The following lemma allows us to complete an embedding of almost all of $T$ into $H$, provided that enough appropriate absorbing tuples exist.
\begin{lemma}[Absorbing Lemma]
\label{lemma:absorbing_lemma}
Let $1/n \ll \delta \ll \alpha < 1/\Delta$ with $n$ odd.
Let $T$ be a $3$-tree on $n$ vertices of maximum degree $\Delta$ with a valid ordering of the edges $e_1, \dots, e_{(n-1)/2}$ and let $T_0 = \{e_1, \dots, e_{(n'-1)/2}\}$ be a subtree of $T$ on $n' \geq (1-\delta)n$ vertices, for $n'$ odd. Let $H$ be a $3$-graph on $n$ vertices, and $\phi_0$ be an embedding $\phi_0: V(T_0) \rightarrow V(H)$. Suppose $\cA \subseteq A_{\Delta}(H)$ is a family of pairwise vertex-disjoint $\Delta$-absorbing tuples in $H$ such that
\begin{enumerate}
    \item every tuple in $\cA$ is inhabited by $\phi_0$ and
    \item $|A_{\Delta}(w_1,w_2,w_3) \cap \cA| \geq \alpha n$ for every triple $(w_1, w_2, w_3)$ of distinct vertices of $H$.
\end{enumerate}
Then there exists an embedding of $T$ in $H$.
\end{lemma}
\begin{proof}
Let $V(H) \setminus V(\phi_0(T_0)) = \{x_1, x_2, \dots, x_{n-n'}\}$.
For each even $i\in [n-n']$, let $T_i = \{e_1, \dots, e_{(n'+i-1)/2}\}$ and we define an embedding $\phi_i : T_i \rightarrow H$ and a subset $\cA_i \subset \cA$ such that 
\begin{enumerate}[label=(\alph*)]
    \item $V(\phi_i(T_i)) = V(\phi_0(T_0)) \cup \{x_1, \dots, x_i\}$, \label{cond_emb_1}
    \item $|\cA_i| \le 2i$, \label{cond_emb_2}
    \item every tuple in $\cA \setminus \cA_i$ is inhabited by $\phi_i$. \label{cond_emb_3}
\end{enumerate}
The final embedding of $T$ in $H$ will then be given by $\phi_{n-n'}$.

These conditions clearly hold for $\phi_0$. Suppose we have defined some $\phi_i$ and $\cA_i$ for some even $i \le n-n'-2$ such that the conditions above hold. Let $e_{(n'+i+1)/2} = \{z_1, z_2, z_3\} \in E(T)$ for some $z_1 \in V(T_i)$ and $z_2,z_3 \notin V(T_i)$ be the next edge we are to embed. Set $w_1 := \phi_i(z_1)$ and $w_2 := x_{i+1}, w_3 := x_{i+2}$. Note that since $|A_{\Delta}(w_1, w_2, w_3) \cap (\cA \setminus \cA_i)| \geq \alpha n - 2i \geq (\alpha-2\delta)n > 0$, there exists a $\Delta$-absorbing tuple $(S_{v_2}, S_{v_3}) \in \cA \setminus \cA_i$ for $(w_1, w_2, w_3)$. For $j\in\{2,3\}$, let $S_{v_j} = \big( v_j, (u^{j,h}_1, u^{j,h}_2)_{h \in[\Delta]} \big)$. Let $\cA_{i+2} := \cA_i \cup \{(S_{v_2}, S_{v_3}), (S_{v_3},S_{v_2})\}$ and define $\phi_{i+2}$ as follows:
$$\phi_{i+2}(v) = \begin{cases}
  w_j  & \mbox{if }v = \phi^{-1}_i(v_j), j \in \{2,3\} \\
  v_j & \mbox{if }v = z_j, j \in \{2,3\} \\
  \phi_i(v) & \mbox{if }v \in V(T_i) \setminus \{\phi_i^{-1}(v_2), \phi_i^{-1}(v_3)\}.
\end{cases}$$
Note that $\phi_{i+2}$ is injective and conditions \ref{cond_emb_1}, \ref{cond_emb_2}, and \ref{cond_emb_3} hold by construction.

It remains to check that $\phi_{i+2}$ is indeed an embedding. The only edge in $E(T_{i+2}) \setminus E(T_{i})$ is $e_{(n'+i+1)/2} = \{z_1, z_2, z_3\}$, which maps to $\{w_1, v_2, v_3\} \in E(H)$ since $(S_{v_2}, S_{v_3})$ is a $\Delta$-absorbing tuple for $(w_1,w_2,w_3)$. For each edge $e$ of $T_i$ that does not contain $\phi_i^{-1}(v_2)$ or $\phi_i^{-1}(v_3)$, we have $\phi_i(e) = \phi_{i+2}(e)$. Note that no edge of $T_i$ can contain both $\phi_i^{-1}(v_2)$ and $\phi_i^{-1}(v_3)$, since $(S_{v_2}, S_{v_3})$ is inhabited by $\phi_i$. Now suppose $e \in E(T_i)$ is such that $\phi_i^{-1}(v_j) \in e$, for $j \in \{2,3\}$. Observe that $\big( w_j, (u^{j,h}_1, u^{j,h}_2)_{h \in [\Delta]} \big)$ is a star in $H$ because $(S_{v_2}, S_{v_3})$ is a $\Delta$-absorbing tuple for $(w_1, w_2, w_3)$. Since $(S_{v_2}, S_{v_3})$ is inhabited by $\phi_i$, we have that $\phi_i(e) \in E(S_{v_j})$ and so $\phi_{i+2}(e) = (\phi_i(e) \setminus \{v_j\}) \cup w_j$ is also an edge in $H$.
\end{proof}

The next lemma is used to find a family $\cA$ of $\Delta$-absorbing tuples as required by the absorbing lemma.
\begin{lemma}
\label{lemma:absorbing_set}
Let $1/n \ll \alpha \ll \beta \ll \gamma, 1/\Delta$.
Let $H$ be a $3$-graph on $n$ vertices with minimum degree $(1/2 + \gamma)\binom{n}{2}$. Then there exists a set $\cA\subseteq A_\Delta(H)$ of at most $\beta n$ pairwise vertex-disjoint $\Delta$-absorbing tuples $\{(S_{v^i_2},S_{v^i_3})\}$ such that for every triple $(w_1, w_2, w_3)\in V(H)^3$ we have $|A_{\Delta}(w_1, w_2, w_3) \cap \cA| \geq \alpha n$. 
\end{lemma}
\begin{proof}
We first show that for any triple $(w_1, w_2, w_3)$ in $H$, there are many $\Delta$-absorbing tuples. Fix $(w_1, w_2, w_3)$. For each of the at least $(1/2 + \gamma/2)\binom{n}{2}$ many edges $\{w_1, v_2, v_3\}$ disjoint from $\{w_2,w_3\}$, we have that for each $j\in \{2,3\}$ the common neighbourhood of $w_j$ and $v_j$ has size at least $2\gamma \binom{n}{2}$. Thus, for each $v_j$ there are at least $\left( 2\gamma \binom{n}{2} -6\Delta n \right)^\Delta/\Delta!$ many choices for the star $S_{v_j}$, even if the other star has been fixed. In total, we get at least $(1/2 + \gamma/2) \binom{n}{2} \left( \gamma \binom{n}{2}\right)^{2\Delta}/(\Delta!)^2  \geq  \beta n^{4\Delta + 2}$ many $\Delta$-absorbing tuples for $(w_1,w_2,w_3)$.

Pick a subset $\cA'$ of $A_{\Delta}(H)$ by choosing each of its members independently at random with probability $p = \frac{c}{n^{4\Delta+1}}$, for some $\alpha \ll c \ll \beta$. Note that $|A_{\Delta}(H)| \le n^{4\Delta + 2}$, so by a Chernoff bound, we have that w.h.p. $|\cA'| \le 2c n \le \beta n$. By a Chernoff bound and a union bound, w.h.p. for each $(w_1, w_2, w_3)$ we have that $|A_\Delta(w_1, w_2, w_3)\cap \cA'| \geq \beta c n/2$. Let $X$ count the number of pairs of tuples in $\cA'$ which overlap in at least one vertex. Then
$$ \mathbb{E}[X] \le 2^{10\Delta} (5\Delta)! n^{8\Delta + 3} p^2 \le 2^{10\Delta} (5\Delta)! c^2 n,$$
and so by Markov's inequality, with probability at least $1/2$, we have $X \le 2^{10\Delta+1} (5\Delta)! c^2 n$. Fix an outcome of $\cA'$ for which all these probabilistic events hold. Removing all pairs of overlapping tuples in $\cA'$, we get a set $\cA$ of at most $\beta n$ vertex-disjoint tuples such that $|A_\Delta(w_1, w_2, w_3)\cap \cA'| \geq \beta c n/2 - 2^{10\Delta+1} (5\Delta)!c^2 n \geq \alpha n$ for each $(w_1,w_2,w_3)$, as desired.
\end{proof}

To satisfy the conditions of the absorbing lemma, we also need to ensure that the embedding of almost all of $T$ inhabits every tuple in $\cA$. The next lemma takes care of that.
\begin{lemma}(Inhabiting Lemma)
\label{lemma:inhabiting_lemma}
 Let $1/n \ll \beta \ll \nu \ll \gamma, 1/\Delta$.
 Let $H$ be a $3$-graph on $n$ vertices with minimum degree $(1/2 + \gamma)\binom{n}{2}$ and let $T$ be a loose tree on $\nu n$ vertices with maximum degree $\Delta$. Let $\cA \subset A_{\Delta}(H)$ be a set of at most $\beta n$ pairwise vertex-disjoint $\Delta$-absorbing tuples of $H$. Then there is an embedding $\phi:V(T) \rightarrow V(H)$ such that every $\Delta$-absorbing tuple in $\cA$ is inhabited by $\phi$.
\end{lemma}
\begin{proof}
Let $m:=2|\cA|\le 2\beta n$ and let $S_1, S_2, \dots, S_m$ be the vertex-disjoint stars from $\cA$. We consecutively define subtrees $T_1, \dots, T_m$ of $T$ and embeddings $\phi_1, \dots, \phi_m$ such that $\phi_i : T_i \rightarrow H$ is an embedding of $T_i$ in $H$ and $S_1, \dots, S_i$ are inhabited by $\phi_i$ for each $i \in [m]$. In addition, we do this in such a way that $v(T_i) \le (2\Delta+4) i$. Define $T_1$ to be an arbitrary vertex $r \in V(T)$ along with all edges incident to $r$. Thus, $T_1$ is a star of degree at most $\Delta$. Set $\phi_1(r) = v$ where $v$ is the center of $S_1$, and for each $e\in E(T_1)$, set $\phi_1(e)$ to be a distinct edge of $S_1$. Note that $S_1$ is inhabited by $\phi_1$ and $v(T_1) \le 2\Delta + 1$.

Suppose $T_1, \dots, T_i$ and $\phi_1, \dots, \phi_i$ have been defined for some $i \in [m-1]$. Note that there must be some vertex $x \in V(T) \setminus V(T_i)$ such that the shortest loose path from $x$ to $T_i$ has length $3$ (in terms of edges), since otherwise all vertices in $T$ are reachable by a path of length $2$ from $T_i$, implying that
$$v(T) \le v(T_i)(1 + 2\Delta + 4\Delta^2) \le 8\Delta^2(2\Delta +4)i \le 8\Delta^2(2\Delta +4)2\beta n \ll \nu n = v(T),$$
leading to a contradiction. Let $T_{i+1}$ be the union of $T_i$, the shortest loose path $P$ from $T_i$ to $x$, and all edges incident to $x$ in $T$. Thus,
$$v(T_{i+1}) \le v(T_i) + v(P\setminus T_i) + (\Delta -1)2 \le (2\Delta + 4)i + 6 + 2\Delta - 2\le(2\Delta + 4)(i+1).$$
We now define $\phi_{i+1}$ by setting $\phi_{i+1}(v) = \phi_i(v)$ for each $v \in V(T_i)$, and extending it to $T_{i+1}$ as follows. Let $w$ be the center of $S_{i+1}$ and $\{w,u_1,u_2\}$ be one of the edges of $S_{i+1}$. Suppose $P = (v_1, v_2, v_3, v_4, v_5, v_6, v_7 = x)$, where $v_1 \in V(T_i)$. Set $\phi_{i+1}(x) = w$ and $\phi_{i+1}(v_5) = u_1$ and $\phi_{i+1}(v_6) = u_2$. Let $\phi_{i+1}$ map each edge $e \in E(T_{i+1}) \setminus \{v_5, v_6, x\}$ with $x \in e$ to a distinct edge in $E(S_{i+1}) \setminus \{w, u_1, u_2\}$. Finally, we find a loose path of length $2$ in $H \setminus (\phi_i(T_i) \cup V(S_{i+1}))$ between $\phi_{i+1}(v_1)$ and $\phi_{i+1}(v_5) = u_1$, and we let $\phi_{i+1}$ map $(v_1, v_2, v_3, v_4, v_5)$ to that path, thereby finishing the definition of $\phi_{i+1}$. It remains to show that such a path exists.

Let $y:= \phi_{i+1}(v_1)$ and $z:= \phi_{i+1}(v_5)$, and set $H':= H \setminus (\phi_i(T_i) \cup V(S_{i+1})) \cup \{y,z\}$. We would like to show that there are edges $e, f \in E(H')$ with $y \in e \setminus f$, and $z \in f \setminus e$ and such that $|e \cap f| = 1$. Note that $$\delta_1(H') \geq \delta_1(H) - |\phi_i(T_i) \cup V(S_{i+1})|n \geq (1/2 + \gamma) \binom{n}{2} - \nu n^2 - (2\Delta+1)n \geq (1/2 + \gamma/2) \binom{n}{2}.$$ Then there are at least $\gamma/2\binom{n}{2}$ pairs of vertices in $H'$ that form an edge with both $y$ and $z$. By averaging, there is a vertex $b\in H'$ which is in at least $\gamma n/4 \ge 2$ of these pairs, say $\{a,b\}$ and $\{b,c\}$. Then setting $e=\{y,a,b\},f=\{z,c,b\}$ gives the desired path. 

Having finished the definition of $\phi_{i+1}$, note that $\phi_{i+1}$ is an embedding of $T_{i+1}$ into $H$, and $S_1, \dots, S_{i+1}$ are inhabited by $\phi_{i+1}$ by construction. At the end of this process, we have a subtree $T_m \subset T$ and an embedding $\phi_m:T_m \rightarrow H$ that inhabits all tuples in $\cA$. We then extend $\phi_m$ to an embedding of all of $T$ into $H$ by embedding the extra edges one by one. This is possible because there exists a valid ordering of the edges of $T$ such that $T_m$ consists of a prefix of that ordering, so we can follow the ordering until the end. At every step, the minimum degree of any vertex of $H$ into the set of unoccupied vertices is at least $(1/2 + \gamma) \binom{n}{2} - \nu n^2 \geq (1/2 + \gamma/2) \binom{n}{2}$, so an embedding of the next edge can be chosen greedily.
\end{proof} 

\section{Proof of the main result}\label{sec:proof}
We are now ready to prove our main result.
\begin{proof}[Proof of Theorem~\ref{thm:main}]
Let $1/n \ll \delta \ll \alpha \ll \beta \ll \nu \ll \gamma, 1/\Delta$, where $n$ is odd.
Let $H$ be a $3$-graph on $n$ vertices with $\delta_1(H) \geq (5/9+\gamma)\binom{n}{2}$, and let $T$ be a loose $3$-tree with $\Delta_1(T) \le \Delta$. We show that $T$ can be embedded into $H$.

We first apply Lemma~\ref{lemma:absorbing_set} to get a set $\cA$ of at most $\beta n$ pairwise vertex-disjoint pairs of stars $\{(S_{v_2^i}, S_{v_3^i})\}_i$ such that for every triple $(w_1, w_2, w_3)$ in $H$, we have $|A_{\Delta}(w_1, w_2, w_3) \cap \cA| \geq \alpha n$.

Next, root $T$ arbitrarily at some vertex $r$. Recall that $T(x)$ denotes the subtree of $T$ rooted at $x$. Find a subtree $T(x) \subset T$ of size $\nu n \le v(T(x)) \le 2 \Delta \nu n$. This can be done by setting $x := r$ and, until $x$ has a child $y$ whose subtree has at least $\nu n$ vertices, set $x := y$. At some point this process reaches a vertex $x$ whose subtree $T(x)$ has at least $\nu n$ vertices, but all its children's subtrees have fewer than $\nu n$ vertices, implying that $v(T(x)) \le 2\Delta \nu n$. Let $\nu' := v(T(x)) / n$ and apply Lemma~\ref{lemma:inhabiting_lemma} with $\nu:= \nu'$ and $T:=T(x)$ to find an embedding $\phi:V(T(x)) \rightarrow V(H)$ such that every $\Delta$-absorbing tuple in $\cA$ is inhabited by $\phi$. Suppose $\phi(x) = z \in V(H)$.

Now let $H' := H \setminus (\phi(V(T(x))) \setminus \{z\})$ and note that $\delta_1(H') \geq (5/9 + \gamma/2)\binom{n}{2}$. Let $T' := T \setminus (T(x) \setminus \{x\})$ and root $T'$ at $x$. Remove leaf edges from $T'$ repeatedly to get $T''$ such that $v(T') - v(T'') = \delta n$. Apply Lemma~\ref{lemma:almost_embedding} with $G:= H'$, $T:=T''$, $r:=x$ and $x:=z$ to find an embedding $\phi''$ of $T''$ into $H'$ with $\phi'(x) = z$.

Finally, let $T_0 := T(x) \cup T''$ and note that $v(T_0) = (1- \delta)n$. Combine $\phi$ and $\phi''$ into an embedding $\phi_0$ of $T_0$ into $H$, which can be done since $\phi(T(x)) \cap \phi''(T'') = \phi(x) = \phi''(x)$. Then the tree $T_0$, the embedding $\phi_0$, and the set of $\Delta$-absorbing tuples $\cA$ satisfy the conditions of Lemma~\ref{lemma:absorbing_lemma}, which we can apply to get an embedding of $T$ in $H$.
\end{proof}

\section{Concluding remarks}\label{sec:concluding}

As mentioned in the introduction, the main result in \cite{kss1995} was later generalised by the same authors \cite{kss2001} to trees with maximum degree $c n/\log n$ for some small constant $c$. This is tight up to the constant $c$. It is natural to ask whether Theorem~\ref{thm:main} also holds for trees of maximum degree larger than constant, and if so, how high the maximum degree can be.

In proving Theorem~\ref{thm:main}, we reduced the problem of finding a tree in our $3$-graph $H$ to finding a tight Hamilton cycle $\cC$ in the reduced graph $\cR$ of a regular partition of $H$. To do this we applied a known result on the degree threshold for the existence of a tight Hamilton cycle (Theorem~\ref{thm:degree-condition-tight-Hamiltonicity}, proved in \cite{reiher2019minimum}). However, most vertices of $T$ were embedded into the edges of a perfect matching $M\subseteq \cC$ obtained by taking every third edge of $\cC$. This motivates us to ask whether the application of Theorem~\ref{thm:degree-condition-tight-Hamiltonicity} can be replaced with the analogous result for perfect matchings. If this is the case, the following transference principle should hold for hypergraphs of higher uniformity.

\begin{conjecture}
\label{conj:transference}
Let $1 \leq d < k$ and suppose $\delta_{k,d}^{PM}$ is the minimum $d$-degree threshold for the existence of a perfect matching in $k$-uniform hypergraphs.\\
Then for all $\gamma>0$ and $\Delta\in \N$ there exists $n_0\in \N$ such that any $k$-graph $H$ on $n\ge n_0$ vertices with $\delta_d(H)\ge(\delta_{k,d}^{PM}+\gamma)\binom{n}{k-d}$ contains every loose spanning tree $T$ with $\Delta_1(T)\le \Delta$.
\end{conjecture}

This conjecture holds for $d=k-1$ by the main result in \cite{pavez2024dirac}. In this paper we showed that it also holds for $(d,k)=(1,3)$.
Note that, as discussed in~\cite{lang2022minimum}, the minimum $d$-degree threshold for containment of tight Hamilton cycles is strictly larger than $\delta_{k,d}^{PM}$ when $k-d$ is large. This implies that to prove Conjecture~\ref{conj:transference} in full generality using the Regularity Lemma, one has to move away from relying on tight Hamilton cycles in the cluster graph.

There are other notions of a hypertree for which considering minimum degree conditions would be interesting, such as Berge trees, which are more general than $\ell$-trees, and $k$-expansion trees, which are a special case of loose trees (see e.g.~\cite{stein2020tree}). Every instance of the latter can be obtained from a $2$-tree by adding $k-2$ new vertices to each edge. It seems that for $k$-expansion trees, the threshold should be the same as for loose Hamilton cycles. We believe that our techniques can be extended in a straightforward way to show this.
\begin{conjecture}
Let $1 \leq d < k$ and suppose $\delta_{k,d}^{LHC}$ is the minimum $d$-degree threshold for the existence of a loose Hamilton cycle in $k$-uniform hypergraphs.\\
Then for all $\gamma>0$ and $\Delta \in \mathbb{N}$ there exists $n_0 \in \mathbb{N}$ such that any $k$-graph $H$ on $n \geq n_0$ vertices with $\delta_d(H) \geq (\delta^{LHC}_{k,d} + \gamma)\binom{n}{k-d}$ contains every $k$-expansion spanning tree $T$ with $\Delta_1(T) \leq \Delta$.
\end{conjecture}

\bibliographystyle{habbrv}
\bibliography{ref}

\end{document}